\documentclass[12pt]{amsart}
\usepackage[colorlinks=true,citecolor=magenta,linkcolor=magenta]{hyperref}
\usepackage{amssymb,verbatim,amsthm, amscd,amsmath,graphicx,mathtools}
\usepackage[capitalize]{cleveref}
\usepackage{enumerate}
\usepackage{graphicx}
\usepackage{caption}
\usepackage{subcaption}
\usepackage{pdfsync}
\usepackage{color,colortbl}
\usepackage{tikz-cd, tikz-3dplot}
\usepackage{fullpage}
\usepackage{bbm}
\usepackage{comment}

\numberwithin{equation}{section}
\newtheorem{theorem}{Theorem}[section]
\newtheorem{proposition}[theorem]{Proposition}
\theoremstyle{definition}
\newtheorem{definition}{Definition}[section]
\newtheorem{example}[theorem]{Example}
\newtheorem{remark}[theorem]{Remark}
\newtheorem{lemma}[theorem]{Lemma}
\newtheorem{corollary}[theorem]{Corollary}

\newcommand{\la}{\langle}
\newcommand{\ra}{\rangle}

\DeclareMathOperator{\soc}{socdeg}
\DeclareMathOperator{\hilb}{Hilb_{\A}}
\DeclareMathOperator{\rank}{rank}

\def\k{{\Bbbk}}
\def\D{{\Delta}}
\def\A{{A(\D)}}
\def\l{{\ell}}

\title{On the Lefschetz property for quotients by monomial ideals containing  squares of variables}
\author{Hailong Dao}  
\address{Department of Mathematics, University of Kansas, Lawrence, KS 66045-7523, USA}
\email{hdao@ku.edu}
\urladdr{https://www.math.ku.edu/~hdao/}
\author{Ritika Nair}
\address{Department of Mathematics, University of Kansas, Lawrence, KS 66045-7523, USA}
\email{rnair@ku.edu}

\begin{document}

\maketitle

\begin{abstract}
    Let $\Delta$ be an (abstract) simplicial complex on $n$ vertices. One can define the Artinian monomial algebra $\A = \k[x_1, \ldots, x_n]/ \la x_1^2, \ldots, x_n^2, I_{\D} \ra$, where $\k$ is a field of characteristic $ 0$ and $I_\D$ is the Stanley-Reisner ideal associated to $\D$. In this paper, we aim to characterize the Weak Lefschetz Property (WLP) of $A(\Delta)$ in terms of the simplicial complex $\Delta$. We are able to completely analyze when  WLP holds in degree $1$, complementing work by Migliore, Nagel and Schenck in \cite{MNS2020}. We give a complete characterization of all $2$-dimensional pseudomanifolds $\D$ such that $\A$ satisfies WLP. We also construct Artinian Gorenstein algebras that fail WLP by combining our results and the standard technique of Nagata idealization. 
\end{abstract} 

\section{Introduction}

Let $A$ be an Artinian algebra over a field $\k$ of characteristic $0$. $A$ is said to satisfy the Weak Lefschetz Property ( WLP for short) if the map between graded pieces $A_i \to A_{i+1}$ of $A$ induced by multiplying by a general linear form has full rank for all $i$.
This condition mimics the property of the cohomology ring of a projective variety where the Lefschetz's hyperplane Theorem holds, hence the name.

The last few decades have seen tremendous growth in the study of WLP, due to its many fascinating connections to other areas of mathematics:  Dilworth number of posets, Schur-Weyl duality, the $g$-conjecture and its generalizations, syzygy bundles, Laplace equations and Togliatti systems, just to name a few.(See \cite{APP2021, BK2007, HMMNWW2013,IMS2022, MMO2013, MM2016, MMN2011, S1983}.)

Even in the case where $A$ is defined by a {\it monomial ideal}, the literature is quite extensive, see for instance \cite{AB2020, CN2016, HMMNWW2013, MPS2008, MMO2013, MM2016, MMN2011, MNS2020} and the references therein. In this work, we focus on a special class of Artinian monomial ideal. Let $J\subset S=  \k[x_1, \ldots, x_n]$ be a monomial ideal containing squares of all the variables. Clearly, $J= \la x_1^2, x_2^2, \ldots, x_n^2, I \ra$
where $I$ is a {\it square-free} monomial ideal. Thus, we can write $I$ as the Stanley-Reisner ideal of a simplicial complex $\D$. Because of this correspondence, we shall write $J_{\D}$ for $J$ and $\A$ for the Artinian quotient $S/J_{\D}$.

Information about $I$ and $J_{\D}$ such as Hilbert functions, Betti numbers, type, regularity, can then be conveniently studied via combinatorial and topological properties of $\Delta$, see \cite[Theorem 2.1]{MPS2008}, \cite[Lemma 2.1]{MNS2020} and \cref{inv_A} below.  We note also that the polarization of $J_{\D}$ has been utilized recently in \cite{DV2021} to study the Koszul properties of quadratic monomial ideals. 

In this paper, we study  WLP of $\A$ in degree $1$ and $d=\dim \D$, when $\k$ is a field of characteristic $0$ (our results also hold if the characteristic of $\k$ is large enough, see \cite[Proposition 7.2]{MN2013}). Our first main result completely characterizes  WLP in degree $1$, complementing the work in \cite{MNS2020}.

\begin{theorem}(Theorem \ref{deg1}) 
Let $\D$ be a simplicial complex, $G(\D)$ the $1$-skeleton of $\D$, and $A=A(\D)$ the Artinian ring defined by the Stanley-Reisner ideal of $\D$ plus the squares of all variables. 
 \begin{enumerate}
         \item[(i)] If $\dim_\k A_2 \ge \dim_\k A_1$, then  WLP holds in degree $1$ if and only if $G(\D)$ has no bipartite components.
         \item[(ii)] If $\dim_\k A_2 < \dim_\k A_1$, then  WLP holds in degree $1$ if and only if each bipartite component of $G(\D)$ (if it exists) is a tree and each non-bipartite component satisfies the property that the number of edges in the component is equal to the number of vertices in the component. Further, in this case, WLP holds in degree $1$ implies WLP holds in all degrees.
     \end{enumerate}
\end{theorem}

As an application, we show that if $J$ is a  monomial ideal containing the squares of variables with Artinian quotient $A$, and the multiplication map by a general linear form $A_1\to A_2$ is surjective, then the regularity of $J$ is at most $4$ (See \cref{soc_3}). This is related to the work by Eisenbud-Huneke-Ulrich on regularity of Tor in \cite{EHU2006} and the $p$-basepoint freeness  condition and Hankel index studied by Blekherman-Sinn-Velasco \cite{BSV2017}. See \cite[Definition 2.3, Theorem 2.4 in Section 2.2]{BSV2017}.

Finally, we use similar ideas to establish  WLP in degree $d=\dim \D$ using the dual graph of $\D$ (see \cref{dim2_charc}). In this direction, we have (see \cref{psd_wo_bdy} and \cref{psd_with_bdy}):

\begin{theorem} 
For $A=\A$ corresponding to a $d$-dimensional pseudomanifold $\D$,  WLP holds in degree $d$ if and only if:
\begin{enumerate}
    \item[(i)] $\D$ has boundary or
    \item[(ii)] $\D$ has no boundary, and the dual graph of $\D$ is not bipartite.
\end{enumerate}
\end{theorem}

The above results combine to give the following corollary: when $\D$ is a two-dimensional pseudomanifold,  WLP fails for $A(\D)$ if and only $\D$ has no boundary and the dual graph of $\D$ is bipartite (\cref{dim2_charc}). Furthermore, when $\D$ is a planar triangulation without boundary, the dual graph is not bipartite if and only if the graph of $\D$ is not Eulerian. (\cref{eulerian}). Moreover, the first barycentric subdivisions of triangulations of a two-dimensional pseudomanifold without boundary always give $A(\D)$ that fail WLP (\cref{bary}). \\

In Section 5, we construct Artinian Gorenstein algebras that fail WLP using  our previous results and the technique of Nagata idealization inspired by \cite{MSS2021, MS2020} (see \cref{Gor_ex}). In particular, we construct Artinian Gorenstein algebra in $2n$ variables that fail WLP in degree $1$ and having unimodal Hilbert series, for each $n \ge 4$.\\

\noindent\textbf{Acknowledgements}: We thank Aaron Dall, Alessio D'Alì, Serge Lawrencenko, Hal Schenck and Alexandra Seceleanu and an anonymous referee for some helpful correspondence on the topics of this work.

\section{Preliminary definitions}

Let $V$ be a finite non-empty set of $n$ vertices. An \emph{(abstract) simplicial complex} $\D$ is a non-empty subset of the power set of $V$ such that $F \subseteq E, E \in \D \implies F \in \D$. The elements of $\Delta$ consisting of $k+1$ vertices are called \emph{$k$-faces} or \emph{$k$-simplices} (dimension of the face is $k$). The faces which are maximal with respect to inclusion are called \emph{facets}. The dimension of $\Delta$, $\dim(\Delta) = max \{ \dim(F) : F \text{ is a facet of } \Delta \}$. $\Delta$ is said to be \emph{pure} if every facet has the same dimension. The \emph{$d$-skeleton} of $\D$, denoted by $\D_d$, is the set of all faces of dimension at most $d$. The \emph{$f$-vector} of $\D$, $f(\D) = (f_{-1}(\D), f_0(\D), \ldots, f_{\dim(\D)}(\D))$ is defined as $f_i(\D) = $ number of $i$-faces of $\D$, with $f_{-1}(\D) = 1$ ($\emptyset \in \D$). We define $\Delta$ by defining the facets. \\
Let $S = \k[x_1, \ldots, x_n]$ be the polynomial ring in $n$ variables over a field $\k$ of characteristic $0$. For a given $\D$, define the ideal 
 $ I_\Delta = \la x_{i_1} \cdots x_{i_m} : \{i_1 \ldots i_m \} \notin \Delta \ra$. Then, $\k[\Delta] = S/I_\Delta$ is the \emph{Stanley-Reisner ring} (or face ring) of $\D$. We further define the ideal $J_\D = \la x_1^2, x_2^2, \ldots, x_n^2, I_\Delta \ra$ and consider the quotient ring $A(\D) = S/J_\D$. \\

\begin{example}
$\Delta = \la123, 134, 45\ra$. The non-empty faces of $\D$ are \\$\{1\}, \{2\}, \{3\}, \{4\}, \{5\}, \{12\}, \{13\}, \{23\}, \{14\}, \{34\}, \{45\}, \{123\},  \{134\}$.
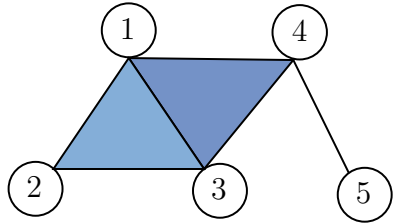
\begin{center}

\tikzset{every picture/.style={line width=0.75pt}} 

\begin{tikzpicture}[x=0.75pt,y=0.75pt,yscale=-1,xscale=1]

\draw  [fill={rgb, 255:red, 116; green, 146; blue, 199 }  ,fill opacity=1 ] (176,206) -- (221,151) -- (138,150) -- cycle ;
\draw  [color={rgb, 255:red, 0; green, 0; blue, 0 }  ,draw opacity=1 ][fill={rgb, 255:red, 132; green, 174; blue, 216 }  ,fill opacity=1 ] (138,150) -- (176,206) -- (100,206) -- cycle ;
\draw    (221,151) -- (249,208) ;

\draw  [line width=0.75]   (138, 136) circle [x radius= 13.6, y radius= 13.6]   ;
\draw (132.07,128.35) node [anchor=north west][inner sep=0.75pt]  [rotate=-0.5]  {$1$};
\draw    (91, 216) circle [x radius= 13.6, y radius= 13.6]   ;
\draw (85,208.4) node [anchor=north west][inner sep=0.75pt]    {$2$};
\draw    (184, 217) circle [x radius= 13.6, y radius= 13.6]   ;
\draw (178,209.4) node [anchor=north west][inner sep=0.75pt]    {$3$};
\draw    (224.22, 137) circle [x radius= 13.72, y radius= 13.72]   ;
\draw (224.37,129.4) node [anchor=north] [inner sep=0.75pt]  [xslant=0.02]  {$4$};
\draw    (257, 219) circle [x radius= 13.6, y radius= 13.6]   ;
\draw (251,211.4) node [anchor=north west][inner sep=0.75pt]    {$5$};

\end{tikzpicture}

\end{center}
 $I_{\D} = \la x_1x_5, x_2x_4, x_2x_5, x_3x_5 \ra \subseteq S = \k[x_1, x_2, x_3, x_4, x_5]$ and $\k[\Delta] = S/\la x_1x_5, x_2x_4, x_2x_5, x_3x_5 \ra$. $J_{\D} = \la x_1^2, x_2^2, x_3^2, x_4^2, x_5^2, x_1x_5, x_2x_4, x_2x_5, x_3x_5 \ra$ and $\A = S/ \la x_1^2, x_2^2, x_3^2, x_4^2, x_5^2, x_1x_5, x_2x_4, x_2x_5, x_3x_5 \ra$.
 
 \end{example}

\begin{definition} \label{soc}
For a standard graded Artinian ring $A$, \emph{socle} of $A$ is given by $(0:A_{n \ge 1})$.We call any highest degree monomial of the ring a \emph{top socle}, and the degree of this monomial is the \emph{(top) socle degree}, denoted by $\soc(A)$. If $(0:A_{n \ge 1}) = A_{\soc(A)}$, then $A$ is a \emph{level} $\k$-algebra and, if $\dim_{\k}(0:A_{n \ge 1})=1$, $A$ is \emph{Gorenstein}.
\end{definition} 

By labelling the vertices of $V$ by $x_i$, we see that each face $F = \{i_1i_2 \ldots i_k\}$ of $\D$ represents a monomial $x_F = x_{i_1}x_{i_2}\cdots x_{i_k}$ in $A(\D)$. This gives us the following result. \\

\begin{proposition} ~ \label{inv_A} 
    \begin{enumerate}
        \item[(i)] The monomials in $\A$ are in one-one correspondence with the faces of $\D$. For $i>0$, $A(\D)_i$ is a $\k$-vector space with a basis given by $\{ x_F: F \text{ is an }(i-1)-\text{face of } \D \}$.  
        \item[(ii)] The Hilbert Series of $\A$, $\hilb(t) = \sum_{i\ge 0} f_{i-1}(\D)\ t^i$.
        \item[(iii)] $A(\D_i) = \bigoplus_{j \leq i+1}(\A)_j$.
      \item[(iv)] A top socle of $\A$ is given by $x_F$, where $F$ is any facet of $\D$ of dimension $\dim(\D)$. The top socle degree, $\soc(\A) = \dim(\D) + 1$. Note that this number is the (Castelnuovo-Mumford) regularity of $\A$. 
    \end{enumerate}
\end{proposition}

The following definition is from \cite{MN2013}.

\begin{definition}
Let $A$ be a graded Artinian algebra and $\ell$ be a general linear form.\\
$A$ has the \emph{Weak Lefschetz Property (WLP)} if the homomorphism induced by multiplication by $\l$, $\mu_i : A_i \to A_{i+1}$, has maximal rank for all $i$ (i.e., is injective or surjective). \\
Such an $\l$ is called a \emph{Lefschetz element}.

\noindent $A$ is said to have the \emph{Strong Lefschetz Property (SLP)} if the homomorphism induced by multiplication by $\l^j$,
$   \mu_i^j : A_i \to A_{i+j}$, has maximal rank for all $i$ and $j$.\\
Note that unless specified otherwise, $\mu_i = \mu_i^1$.
\end{definition}

The motivation for the study of Lefschetz properties comes from a result first proved by R. P. Stanley in $1980$, which states that SLP holds for $S/I$, where $S=\k[x_1, \ldots, x_n]$ with $char \ \k = 0$ and $I$ is an Artinian monomial complete intersection \cite{S1980}.

\begin{proposition} \label{WLP}  \cite[Proposition 2.2]{MMN2011}
Let $I \subseteq S$ be an Artinian monomial ideal and assume that the field $\Bbbk$ is infinite. Then, $S/I$ has the  WLP iff $\l : = x_1 + \cdots + x_n$ is a Lefschetz element for $S/I$.
\end{proposition}

\begin{remark}
 By part $(iii)$ of \cref{inv_A}, to check for  WLP in degree $i$, i.e.,  $A(\D)_{i} \to A(\D)_{i+1}$, we need to consider only $A(\D_i)$. \\
 Henceforth, we use $\mu_i:A(\D)_i \to A(\D)_{i+1}$ to denote the multiplication map by the linear form $\ell : = x_1 + \cdots + x_n$. We consider the $\k$-bases of $A(\D)_i$ and $A(\D)_{i+1}$, consisting of their monomials in lex order $(x_1 > x_2 > \cdots > x_n)$, to get the matrix representing the map $\mu_i$, and this matrix shall be denoted by $[\mu_i]$.  In terms of the  simplicial complex $\Delta$, by part $(i)$ of \cref{inv_A}, for any $1 \le i \le \dim(\D)$ and an $(i-1)$- face $F$, \\
    \[
        \mu_i(x_F)  = \sum_{\mathclap{\substack{F \subseteq G\\ \dim(G) = \dim(F) + 1}}} \ x_G
   \]
   
 where $x_F$ is the monomial in $\A$ that corresponds to face $F \in \D$.
\end{remark}

\begin{example} 
 Let $I_{\D} \subseteq S = \k[x_1,\ldots, x_7]$ be the edge ideal of path on $7$ vertices and $J_{\D}= \la x_1^2,x_2^2,x_3^2,x_4^2,x_5^2,x_6^2,x_7^2, I_{\D}\ra$. Note that $\D$ here is the independence complex of the path on $7$-vertices. We wish to examine  WLP in degree $1$ for $\A = S/J_{\D}$. Taking $\ell = x_1 + \cdots + x_7$, the associated matrix is given by

 $$  \begin{bmatrix}
 & x_1 & x_2 & x_3 & x_4 & x_5 & x_6 & x_7\\
x_1x_3 & 1 & 0 & 1 & 0 & 0 & 0 & 0\\
x_1x_4 & 1 & 0 & 0 & 1 & 0 & 0 & 0\\
x_1x_5 & 1 & 0 & 0 & 0 & 1 & 0 & 0\\
x_1x_6 & 1 & 0 & 0 & 0 & 0 & 1 & 0\\
x_1x_7 & 1 & 0 & 0 & 0 & 0 & 0 & 1\\
x_2x_4 & 0 & 1 & 0 & 1 & 0 & 0 & 0\\
x_2x_5 & 0 & 1 & 0 & 0 & 1 & 0 & 0\\
x_2x_6 & 0 & 1 & 0 & 0 & 0 & 1 & 0\\
x_2x_7 & 0 & 1 & 0 & 0 & 0 & 0 & 1\\
x_3x_5 & 0 & 0 & 1 & 0 & 1 & 0 & 0\\
x_3x_6 & 0 & 0 & 1 & 0 & 0 & 1 & 0\\
x_3x_7 & 0 & 0 & 1 & 0 & 0 & 0 & 1\\ 
x_4x_6 & 0 & 0 & 0 & 1 & 0 & 1 & 0\\
x_4x_7 & 0 & 0 & 0 & 1 & 0 & 0 & 1\\ 
x_5x_7 & 0 & 0 & 0 & 0 & 1 & 0 & 1\\ 
\end{bmatrix}$$
\\
The above matrix has full rank. Hence,  WLP holds in degree $1$ for $\A$. In general, examining WLP by checking the rank of a matrix can be cumbersome! In \cref{n-path}, we revisit WLP in degree $1$ for $\A$ where $\D$ is the independence complex of $n$-path.
\end{example}

 The following result is a restatement of \cite[Proposition 2.1 (a)]{MMN2011}, in the context of the ring $\A$.
 
\begin{proposition} \label{WLP_extns}
Given a simplicial complex $\D$ of dimension $d$ and the corresponding $\A$, if $\dim_\k \A_i > \dim_\k \A_{i+1}$ for some $i \ge 1$ and if WLP holds in degree $i$, then, WLP holds in all degrees greater than $i$.
\end{proposition} 

\begin{proof}

When $\dim_\k \A_i > \dim_\k \A_{i+1}$ for some $i \ge 1$, if WLP holds in degree $i$ then, $\dim_\k \left(\frac{\A}{\ell \A} \right)_{i+1} = 0$ by surjectivity. So, every monomial in $\A$ of degree $i+1$ is in the ideal $\ell \A$. By \cref{inv_A} (i), for any $j > i$ and monomial $x_E$ of degree $j+1$, the corresponding $j$-face $E$ has some $i$-face $F$, i.e., $x_E \in \la x_F \ra \subseteq \ell \A$. Thus, $\dim_\k \left(\frac{\A}{\ell \A} \right)_{j} = 0$ for all $j > i$. Hence, $\mu_j: \A_j \to \A_{j+1}$ is surjective for all $j \ge i$.
\end{proof}

\section{Weak Lefschetz Property in degree $1$}

In this section, we characterize  WLP in degree $1$ for $\A$ in terms of the $1$-skeleton of $\D$. We denote the $1$-skeleton of $\Delta$ by $G(\D)$, where $V = \{x_1, \ldots, x_n\}$ is the vertex set and $E$ is the edge set consisting of monomial quadrics in $A(\D)$. Let $I(G)$ be the incidence matrix of $G$, i.e., a matrix having dimension $|V| \times |E|$, with $(i,j)^{th}$-entry equal to $1$ if vertex $x_i$ is incident with edge $e_j$ (i.e., vertex $x_i$ divides monomial $e_j$), and zero otherwise. We take the vertices and edges in lex order for writing the incidence matrix.\\

The following result (\cite[Theorem 8.2.1]{GR2001}) is instrumental in proving the results of this paper. For a graph $G$, we define the notation $b_G$ as the number of connected bipartite components.

\begin{proposition} \label{rk}
Given a graph $G$ with $n$ vertices and $b_G$ bipartite connected components, the rank of the incidence matrix  of $G$ over a field of characteristic $0$ is given by $n - b_G$.
\end{proposition}

Let $\mu_1:A(\D)_1 \to A(\D)_2$ be the multiplication map by the linear form $\ell = x_1 + x_2 + \cdots + x_n$\ and $[\mu_1]$ be the matrix representing this map with the monomial $\k$-bases of $\A_1,\ \A_2$ taken in lex order.

\begin{lemma}\label{trnsp}
The matrix $[\mu_1]$ is the transpose of the incidence matrix $I(G)$ of the $1$-skeleton $G(\D)$ of $\D$.
\end{lemma}
\begin{proof}
Given $\D$, by part $(i)$ of \cref{inv_A}, the monomials in $\A_2$ are in one-one correspondence with the edges of $G(\D)$. For a vertex $x_i$,
\begin{align*}
    \mu_1(x_i) & = x_i\cdot \ell  = \sum_{\mathclap{\{x_ix_j\} \ \in \ \D}} x_ix_j
\end{align*}
In the matrix $[\mu_1]$, the column corresponding to $x_i$ will have $1$ in the rows corresponding to quadrics in $\A$ of the form $x_ix_j$ and zero otherwise. Hence, this is the transpose of $I(G)$, where the rows are labelled by the vertices and columns by the edges.\\
\end{proof}

\begin{theorem} ~ \label{deg1}

 \begin{enumerate}
         \item[(i)] If $\dim_\Bbbk \A_2 \ge \dim_\Bbbk \A_1$, then,  WLP holds in degree $1$ if and only if the $1$-skeleton $G(\D)$ of $\D$ has no bipartite components.
         \item[(ii)] If $\dim_\Bbbk \A_2 < \dim_\Bbbk \A_1$, then,  WLP holds in degree $1$ if and only if each bipartite component of $G(\D)$ (if it exists) is a tree and each non-bipartite component satisfies the property that the number of edges in the component is equal to the number of vertices in the component. Further, in this case, WLP holds in degree $1$ implies WLP holds in all degrees.
     \end{enumerate}
\end{theorem}

\begin{proof}
  
By \cref{trnsp}, WLP holds in degree $1$ if and only if $[\mu_1]$ and hence, the incidence matrix $I(G)$ of $G(\D)$ has maximal rank. Note that $n = |V| = \dim_\Bbbk \A_1$ and $|E| = \dim_\Bbbk \A_2$. Hence,  WLP holds in degree $1$ if and only if 
    $$\rank(I(G)) = \min\{\dim_\Bbbk \A_1,\ \dim_\Bbbk \A_2 \}$$ 
    Also, from \cref{rk}, $\rank(I(G)) = n - b_G$
    where $b_G$ is the number of connected bipartite components of $G(\D)$.\\
    We consider the following two cases - 
    \begin{enumerate}
        \item[(i)] $\dim_\Bbbk \A_2 \ge \dim_\Bbbk \A_1:$ Hence, $\rank(I(G)) = \dim_\Bbbk \A_1$, i.e.,  $n - b_G = n$,
           that is to say, $G$ has no bipartite component. 
        
          \item[(ii)]  $\dim_\Bbbk \A_2 < \dim_\Bbbk \A_1:$ Hence, $\rank(I(G)) = \dim_\Bbbk \A_2$, i.e., $|V| - b_G = |E| \iff |V|- |E|=b_G$, i.e.,
            \begin{align*}
               \smashoperator{\sum_{C:\text{ bipartite component}}}\left(|V_C| - |E_C| - 1\right) \ + \ \smashoperator{\sum_{C:\text{ nonbipartite component}}} \left(|V_C| - |E_C|\right)   & = 0
            \end{align*}
        
       \noindent where $|V_C|$ and $|E_C|$ are the number of vertices and edges respectively in the connected component $C$ of the graph $G(\D)$.
             
        For a connected bipartite component $C$, $|E_C| \ge |V_C| - 1$. For each nonbipartite connected component $C$, since it cannot be a tree (as a tree is bipartite), $|E_C| \ge |V_C|$. Since each term in the above sum is non-positive, they become equal to zero. Thus, each bipartite component of $G(\D)$ has $|E_C| = |V_C| - 1$ and so, has to be a tree. Each nonbipartite component of $G(\D)$ satisfies $|E_C| = |V_C|$. \\
        Further, in this case, by \cref{WLP_extns}, if WLP holds in degree $1$, then WLP holds in all degrees.
        \end{enumerate}
           \end{proof}
           
\begin{example} \label{n-path}
  Let $I_\D$ be the edge ideal of a path on $n$ vertices and $\A = S/\la x_1^2, \ldots, x_n^2, I_\D \ra$. $\D$ is the independence complex of the path on $n$-vertices. Let $G(\D)$ be the $1$-skeleton of $\D$. For small values of $n$, it is easy to observe that  WLP holds in degree $1$ for $\A$. For $n \ge 5$, we have $\dim_\Bbbk \A_2 \ge \dim_\Bbbk \A_1$ (since $|E| \ge |V|$) and $G(\D)$ is connected and admits a triangle (an odd cycle), thus, having no bipartite components. Hence,  WLP holds in degree $1$ for all $n$.\\
  
  Given below is $G(\D)$ for $n=5$:

\tikzset{every picture/.style={line width=0.75pt}} 

\begin{tikzpicture}[x=0.75pt,y=0.75pt,yscale=-1,xscale=1]

\draw   (482,115) -- (522,172) -- (442,172) -- cycle ;
\draw    (482,115) -- (532,115) ;
\draw    (532,115) -- (563,147) ;
\draw    (563,147) -- (522,172) ;
\draw    (87,146) -- (139,160) ;
\draw    (42,163) -- (87,146) ;
\draw    (139,160) -- (196,138) ;
\draw    (196,138) -- (244,158) ;

\draw    (481, 102) circle [x radius= 13.6, y radius= 13.6]   ;
\draw (475,94.4) node [anchor=north west][inner sep=0.75pt]    {$1$};
\draw    (575, 153) circle [x radius= 13.6, y radius= 13.6]   ;
\draw (569,145.4) node [anchor=north west][inner sep=0.75pt]    {$2$};
\draw    (434, 183) circle [x radius= 13.6, y radius= 13.6]   ;
\draw (428,175.4) node [anchor=north west][inner sep=0.75pt]    {$3$};
\draw    (538, 102) circle [x radius= 13.6, y radius= 13.6]   ;
\draw (532,94.4) node [anchor=north west][inner sep=0.75pt]    {$4$};
\draw    (524, 185) circle [x radius= 13.6, y radius= 13.6]   ;
\draw (518,177.4) node [anchor=north west][inner sep=0.75pt]    {$5$};
\draw    (32.5, 172.5) circle [x radius= 13.73, y radius= 13.73]   ;
\draw (27,164) node [anchor=north west][inner sep=0.75pt]   [align=left] {1};
\draw    (86.5, 132.5) circle [x radius= 13.73, y radius= 13.73]   ;
\draw (81,124) node [anchor=north west][inner sep=0.75pt]   [align=left] {2};
\draw    (139.5, 174.5) circle [x radius= 13.73, y radius= 13.73]   ;
\draw (134,166) node [anchor=north west][inner sep=0.75pt]   [align=left] {3};
\draw    (194.5, 124.5) circle [x radius= 13.73, y radius= 13.73]   ;
\draw (189,116) node [anchor=north west][inner sep=0.75pt]   [align=left] {4};
\draw    (251.5, 169.5) circle [x radius= 13.73, y radius= 13.73]   ;
\draw (246,161) node [anchor=north west][inner sep=0.75pt]   [align=left] {5};
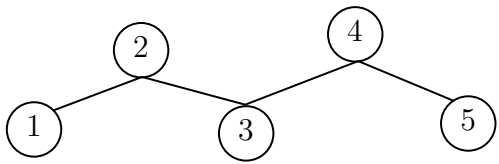
\draw (83,207) node [anchor=north west][inner sep=0.75pt]   [align=left] {Path on $\displaystyle 5$ vertices};
\draw (458,206) node [anchor=north west][inner sep=0.75pt]   [align=left] {$\displaystyle G( \Delta )$};

\end{tikzpicture}

\end{example}

We look at an interesting consequence of \cref{deg1} below.

\begin{definition}
Given a graph $G(V,E)$, any $C \subseteq V$ is a \emph{clique} of $G$ if every pair of vertices in $C$ are adjacent and $\max\{|C| : C \subseteq V \text{ is a clique}\}$ is called the \emph{clique number} of $G$, denoted by  $\omega(G)$.
\end{definition}  
        
\begin{lemma}\label{cliq} 
         $\soc(\A) \le \omega(G(\D))$.
\end{lemma}        
\begin{proof}
          Suppose $x_F = x_{i_1} x_{i_2} \cdots x_{i_j}$ is a top socle of $A(\D)$. By part $(i)$ of \cref{inv_A}, $F$ is a face of $\D$ of dimension $(j-1)$. For every quadric $x_{i_k}x_{i_t}$ that divides $x_F$, there is a corresponding edge in $G(\D)$. Thus, $\{x_{i_1}, x_{i_2},\ldots, x_{i_j}\}$ is a clique of $G(\D)$ and hence, $\soc(\A) = \deg(x_F) \le \omega(G(\D))$.
         
\end{proof}

\begin{corollary} \label{soc_3}
         When $A(\D)$ has  WLP in degree $1$, if $\dim_\Bbbk \A_2 \le \dim_\Bbbk \A_1$, then $\soc(A(\D))$ is at most $3$.
\end{corollary}
\begin{proof}
          When  WLP holds in degree $1$, if $\dim_\Bbbk \A_2 \le \dim_\Bbbk \A_1$, from part $(ii)$ of \cref{deg1}, we see that $G(\D)$ may have bipartite components, in which case the maximum possible clique size is $2$, or every connected nonbipartite component has the property that the number of edges in the component is equal to the number of vertices in the component, allowing clique size at most $3$ since for any clique of size greater than $3$, number of edges is greater than the number of vertices. Hence, the clique number of $G(\D)$ and by \cref{cliq}, $\soc(A(\D))$  is at most $3$.
\end{proof}

  \section{Weak Lefschetz Property in degree $d$ for Pseudomanifolds}
     In this section, we look at the Weak Lefschetz Property in degree $d$ of $\A$ corresponding to the $d$-dimensional pseudomanifold $\D$.    
        \begin{definition} \label{psd_man}
A combinatorial $d$-dimensional \emph{pseudomanifold} is a simplicial complex such that :\\
$(i)$ each facet is a $d$-simplex.\\
$(ii)$ every $(d-1)$-simplex is a face of at least one and at most two $d$-simplices for $d > 1$.\\
$(iii)$ given any two $d$-simplices $F_1, F_k$, there exists a chain of $d$-simplices $F_1, F_2, F_3, \ldots, F_k$ such that $F_i \cap F_{i+1}$ is a $(d-1)$-simplex for $1 \le i \le k-1$.
\end{definition}

We refer to a $d$-dimensional pseudomanifold to be one with boundary if there exists at least one \emph{boundary $(d-1)$-simplex}, i.e., a $(d-1)$-simplex that is a face of exactly one $d$-simplex. We refer to a $d$-simplex having a boundary $(d-1)$-simplex as a \emph{boundary $d$-simplex}. A $d$-dimensional pseudomanifold having no boundary $d$-simplex is called a \emph{pseudomanifold without boundary}.

Given a pure $d$-dimensional complex $\Delta$, any face of dimension $(d-1)$ is called a \emph{ridge}. The \emph{dual graph} $G^*(\D)$ of a pure complex $\Delta$ is the graph in which every facet of $\Delta$ becomes a vertex and two vertices in $G^*(\D)$ are adjacent (i.e., are connected by an edge) if and only if the corresponding facets share a ridge.

\begin{remark}\label{rid-fac}
For $\A$ corresponding to a $d$-dimensional pseudomanifold $\D$,\\
$\dim_\k \A_d = $ number of ridges and $\dim_\k \A_{d+1}= $ number of $d$-simplices. Further, $\A$ is a level algebra.
\end{remark}

\subsection{Pseudomanifold without boundary} 

\begin{theorem}\label{psd_wo_bdy}
For $d \ge 1$, $\A$ corresponding to a $d$-dimensional pseudomanifold without boundary has  WLP in degree $d$ if and only if the dual graph is not bipartite.
\end{theorem}

\begin{proof}
For $d \ge 1$, let $\Delta$ be a $d$-dimensional pseudomanifold without boundary, i.e., every ridge is the face of exactly two $d$-simplices and every $d$-simplex contains $(d+1)$ many ridges. Then,
\begin{equation}\label{psd_dim}
\dim_{\Bbbk}A(\D)_d= \left(\frac{d+1}{2}\right) \cdot \dim_{\Bbbk}A(\D)_{d+1}
\end{equation}
Hence, $\dim_{\Bbbk} A(\D)_d\ge \dim_{\Bbbk} A(\D)_{d+1}$.

Now, consider the dual graph $G^*(\D)$ of $\D$. The number of vertices in $G^*(\D)$, $|V_{G^*(\D)}|$, equals the number of $d$-simplices of $\D$, and the number of edges in $G^*(\D)$, $|E_{G^*(\D)}|$, equals the number of ridges of $\D$. The map $\mu_d:A(\D)_d \to A(\D)_{d + 1}$ then gives the transpose of the map $\mu_1:V_{G^*(\D)} \to  E_{G^*(\D)}$ and hence, by \cref{trnsp},  the corresponding matrix is the incidence matrix $I(G^*(\D))$. So, $\mu_d$ has full rank if and only if $I(G^*(\D))$ has full rank. Note that here, by  \cref{psd_dim},
$$|E_{G^*(\D)}|  \ge |V_{G^*(\D)}|$$
Hence, by part $(i)$ of \cref{deg1}, we see that $\mu_d$ has full rank if and only if the connected graph $G^*(\D)$ is not bipartite.
\end{proof}

\begin{example}
Let $\D$ be the octahedron ($2$-dimensional pseudomanifold without boundary). For the corresponding $\A$,  WLP fails in degree $2$.

\tikzset{every picture/.style={line width=0.75pt}} 

\begin{tikzpicture}[x=0.75pt,y=0.75pt,yscale=-1,xscale=1]

\draw  [fill={rgb, 255:red, 98; green, 219; blue, 192 }  ,fill opacity=1 ] (196.15,108) -- (183.74,195.55) -- (119.51,195.55) -- cycle ;
\draw  [fill={rgb, 255:red, 240; green, 203; blue, 143 }  ,fill opacity=1 ] (196.15,108) -- (268,195.62) -- (184.62,195.62) -- cycle ;
\draw  [fill={rgb, 255:red, 199; green, 143; blue, 211 }  ,fill opacity=1 ] (184.49,271.49) -- (119.51,195.05) -- (184.62,195.11) -- cycle ;
\draw  [fill={rgb, 255:red, 169; green, 200; blue, 236 }  ,fill opacity=1 ] (185.47,271.45) -- (184.62,195.12) -- (267.99,195.62) -- cycle ;
\draw  [dash pattern={on 4.5pt off 4.5pt}]  (118.75,195.04) -- (199.7,180.41) ;
\draw  [dash pattern={on 4.5pt off 4.5pt}]  (196.15,116.69) -- (199.7,180.41) ;
\draw  [dash pattern={on 4.5pt off 4.5pt}]  (199.7,180.41) -- (268,195.62) ;
\draw  [dash pattern={on 4.5pt off 4.5pt}]  (199.7,180.41) -- (184.49,271.49) ;
\draw   (423.06,154.35) -- (464.41,113) -- (560.9,113) -- (560.9,224.62) -- (519.55,265.97) -- (423.06,265.97) -- cycle ; \draw   (560.9,113) -- (519.55,154.35) -- (423.06,154.35) ; \draw   (519.55,154.35) -- (519.55,265.97) ;
\draw    (464.41,113) -- (459.86,221.65) ;
\draw    (459.86,221.65) -- (560.9,223.79) ;
\draw    (423.06,265.97) -- (459.86,221.65) ;

\draw (124.97,270.48) node [anchor=north west][inner sep=0.75pt]   [align=left] {$\displaystyle \Delta $ is the octahedron};
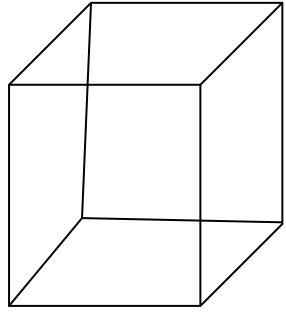
\draw (378.5,269.39) node [anchor=north west][inner sep=0.75pt]   [align=left] {$\displaystyle G^{*}( \Delta )$ is the edge graph of a cube};

\end{tikzpicture}

This can be observed from the fact that the dual graph of the octahedron is the $1$-skeleton (edge graph) of a cube, and is bipartite.\\
\end{example}

\subsection{Pseudomanifold with boundary}~\\

For $d \ge 1$, let $\Delta$ be a $d$-dimensional pseudomanifold with boundary, i.e., there exists at least one boundary $(d-1)$-simplex. 

\begin{theorem} \label{psd_with_bdy}
         WLP in degree $d$ always holds for $\A$ corresponding to a $d$-dimensional pseudomanifold $\D$ with boundary.
\end{theorem}

\begin{proof}
 We use the definition of  WLP to prove this theorem by showing that $\mu_d:A(\D)_d\to A(\D)_{d+1}$ is surjective. We use the notation $x_F$ for the monomial in $\A$ corresponding to the face $F \in \D$.  \\
For a boundary $d$-simplex $F'$, any boundary ridge $B'$ such that $B' \subseteq F'$ gives $\mu_d ( x_{B'}) = x_{F'}$. Now we consider a $d$-simplex $F$ which is not a boundary facet. By part $(iii)$ in \cref{psd_man}, we get some boundary $d$-simplex $F_k$, and a chain of $d$-simplices $F = F_1, F_2, F_3, \ldots, F_k$ such that $F_i \cap F_{i+1} = B_{i}$ is a ridge for $1 \le i \le k-1$. Let $B_k$ be a boundary ridge of $F_k$. \\
 We have $\mu_d ( x_{B_k})  = x_{F_k}$ and $\mu_d ( x_{B_i} ) = x_{F_i} + x_{F_{i+1}}$ for $1 \le i \le k-1$.
 
 Now, define 
\begin{align*}
             x_{M_1} & := x_{B_k}\\
             x_{M_j} & := x_{B_{k-j+1}} - x_{M_{j-1}} \text{ for } 2 \le j \le k\\
         \end{align*}
 Then,
    $\mu_d ( x_{M_1} )   = x_{F_k}$,
   $ \mu_d ( x_{M_2} )   = x_{F_{k-1}}$ and continuing this way, 
    $\mu_d ( x_{M_k} )   = x_{F_1} = x_F$.
Hence, $\mu_d$ is surjective.
\end{proof}

We now look at some interesting corollaries of the above results. 

\begin{corollary}\label{dim2_charc}
Let $\D$ be a $2$-dimensional pseudomanifold. $\A$ has  WLP if and only if $\D$ has boundary or if $\D$ has no boundary, the dual graph $G^*(\D)$ is not bipartite.
\end{corollary}

\begin{proof}
Let $\D$ be a $2$-dimensional pseudomanifold. Note that here, we only need to check for  WLP in degrees $1$ and $2$. Firstly, for degree $1$, the $1$-skeleton $G(\D)$ has only one connected component. Further, each vertex is incident with at least two edges and each edge is incident with two vertices. Thus,
\begin{equation*}
\left(\frac{k}{2}\right) \cdot \dim_{\Bbbk}A(\D)_1  = \dim_{\Bbbk}A(\D)_{2}
\end{equation*}
where $k \ge 2$. Hence, $\dim_{\Bbbk} A(\D)_2 \ge \dim_{\Bbbk} A(\D)_{1}$. Since a $2$-dimensional pseudomanifold has triangular facets making $G(\D)$ nonbipartite, by part $(i)$ of \cref{deg1},  WLP holds in degree $1$ for $\A$.\\
Next, in degree $2$, we see that  WLP holds for $\A$ either if $\D$ has boundary (\cref{psd_with_bdy}) or, if $\D$ has no boundary, the dual graph $G^*(\D)$ is not bipartite (\cref{psd_wo_bdy}).
\end{proof}

\begin{definition}
 A simplicial complex $\D$ is a \emph{triangulation} if there exists a homeomorphism between the geometric realization of $\D$ and a topological manifold.
 In the specific case when $\D$ is a $2$-dimensional pseudomanifold, the facets of $\D$ are triangles ($2$-simplices) and the triangulation is said to be \emph{face $2$-colorable} if it is possible to color all the triangles using $2$-colors such that no two adjacent triangles have the same color.
\end{definition}

  \begin{theorem}
        Let $\D$ be a $2$-dimensional pseudomanifold without boundary that is a triangulation. Then,  WLP holds in degree $2$ for the corresponding $A(\D)$, if and only if the triangulation is not face $2$-colorable.
    \end{theorem}
        
\begin{proof}
         When $\D$ is a $2$-dimensional pseudomanifold without boundary, by \cref{psd_wo_bdy}, $\A$ has  WLP in degree $2$ if and only if the dual graph has no bipartite components. Since the vertices in the dual graph correspond to triangles in the triangulation, this implies that  WLP holds in degree $2$ if and only if the triangulation is not face $2$-colorable.
\end{proof}

\begin{example}
 A tetrahedron is a triangulation without boundary of the sphere $S^2$. It is clearly not $2$-colorable, hence, the corresponding $\A$ satisfies  WLP in degree $2$. 
\end{example}

\begin{remark}
Face $2$-colorability has some interesting connections with the concept of Gr\"unbaum coloring \cite{LVZ2017}. In particular, from \cite[Theorem 1]{LVZ2017}, we see that if the given triangulation of a closed manifold is not Gr\"unbaum colorable, then, it is not face $2$-colorable and hence, WLP holds in degree $2$.
\end{remark}

\begin{definition}
 A graph is said to be \emph{Eulerian} if and only if it has no vertex of odd degree. A \emph{planar triangulation} is a triangulation that can be embedded in the plane.
\end{definition}

 By \cite[Theorem 1.4(b)]{LLCEHYYZ2019}, \cite[Proposition 2]{HSS2020}, a planar triangulation of a graph $G$ is face $2$-colorable if and only if the graph $G$ is Eulerian. This gives the following corollary.

\begin{corollary}\label{eulerian}
        Let $\D$ be a $2$-dimensional pseudomanifold without boundary that is a planar triangulation. Then,  WLP holds in degree $2$ for the corresponding $A(\D)$, if and only if the $1$-skeleton $G(\D)$ is not Eulerian.
\end{corollary}

 \begin{definition}
   \emph{The first barycentric subdivision} of a triangulation is obtained by introducing $4$ new vertices and $6$ new edges in every triangle as follows - the midpoints of the edges and the centroid of the triangle are the new vertices, and the new edges are created by joining the centroid  with the midpoints of the three edges and the vertices of the triangle.
    \end{definition}

 \begin{corollary} \label{bary}
Let $\D$ be a $2$-dimensional pseudomanifold without boundary that is the first barycentric subdivision of a triangulation. Then, WLP in degree $2$ fails for $\A$.
 \end{corollary}
 
\begin{proof}
When $\D$ is a $2$-dimensional pseudomanifold without boundary that is the first barycentric subdivision of a triangulation, by \cite[Corollary 2]{LVZ2017}, we see that the complex is now face $2$-colorable and hence  WLP fails in degree $2$.
\end{proof}     

\begin{example}
For the first barycentric subdivision of a tetrahedron, WLP fails in degree $2$ for $\A$.   
\end{example}

We note that \cite{KE2009} has more results on Lefschetz property of barycentric subdivisions.

\section{Examples of Artinian Gorenstein algebras that fail WLP}

In this section we construct Artinian Gorenstein algebras that fail WLP by combining our results and the standard technique of Nagata idealization. We briefly recall the basic set-up that's relevant for our purpose. Let $\k$ be a field of characteristic $0$. Let $R$ be a standard graded level Artinian $\k$-algebra with socle degree $d$, that is, the socle of $R$ is $R_d$ (see \cref{soc}). Let $\omega$ be the graded canonical module of $R$, and let $\tilde R=R\ltimes \omega(-d-1)$ be the Nagata idealization of $\omega$. Then $\tilde R$ is a standard graded Artinian Gorenstein algebra with socle degree $d+1$, and $\dim_{\k} \tilde R_i = \dim_{\k} R_i+\dim_{\k} R_{d+1-i}$. For more details on this construction we refer to \cite{MSS2021}.   

\begin{proposition}\label{Gor}
Assume the set-up above with $d\geq 2$. Further assume that $\dim_\k R_2+\dim_\k R_{d-1}\geq \dim_\k R_1+\dim_\k R_d$. If $R$ fails surjectivity at degree $d-1$, then $\tilde R$ fails surjectivity, and hence fails WLP in degree $d-1$. 
\end{proposition}

\begin{proof}
Let $l\in \tilde R_1$ be a general linear form. If $l$ induces a surjective map $\tilde R_{d-1} \to \tilde R_d$, then the restriction of $l$ to $R_1$ must induce a surjective map from $R_{d-1} \to R_d$. But indeed no linear form can give surjective map $R_{d-1}\to R_d$ (as general linear forms have maximal possible rank). Finally, to show that failing surjectivity implies failing WLP, we need to show that $\dim_\k \tilde R_{d-1}\geq \dim_\k \tilde R_{d}$, which follows from the assumption on $R$. 
\end{proof}

\begin{corollary}\label{corGor}
Let $d\geq 2$ and $\Delta$ be a $(d-1)$-dimensional pseudomanifold without boundary such that the dual graph of $\Delta$ is bipartite. Let $R=A(\Delta)$. The idealization $\tilde R$ of $R$ is an Artinian Gorenstein ring that fails WLP in degree $d-1$. 
\end{corollary}

\begin{proof}
$R$ is a level algebra (\cref{rid-fac}) that fails surjectivity at degree $d-1$ by \cref{psd_wo_bdy}. Thus, by applying \cref{Gor}, we just need to show that $\dim_\k R_2+\dim_\k R_{d-1}\geq \dim_\k R_1+\dim_\k R_d$.  Note that  $\dim_\k R_i=f_{i-1}(\Delta)$. Clearly $f_1 (\D) \geq f_0 (\D)$: the $1$-skeleton of $\Delta$ is connected and hence $f_1 (\D) \geq f_0 (\D)-1$ with equality if and only if it is a tree, which is impossible as $\Delta$ is a pseudo-manifold.  Equally clearly, $f_{d-2}(\D)=(d/2)f_{d-1}(\D)$: each $(d-2)$-face is in exactly $2$ facets, and each facet has exactly $d$ many $(d-2)$-faces.  Thus the desired inequality follows. 
\end{proof}

\begin{example} \label{Gor_ex}
Let $\D = \{ 12, 23, 34, 14 \}$, a $4$-cycle which is certainly a $1$-dimensional pseudomanifold without boundary, $R = \A = \k[x_1,x_2,x_3,x_4]/J_{\D}$ where $J_{\D} = \la x_1^2, x_2^2, x_3^2, x_4^2, x_1x_3, x_2x_4 \ra$, is a graded Artinian level $\k$-algebra with $\soc(\A) = 2$. By \cref{deg1} (i), since $G(\D)$ is bipartite, WLP fails in degree $1$ (this can also be observed from \cref{psd_wo_bdy}, since the dual graph of $\D$ is also a $4$-cycle which is bipartite).\\
We now refer to \cite[Lemma 3.3]{MSS2021} to compute $\tilde R$, which is a quotient of the polynomial ring $T=\k[x_1,x_2,x_3,x_4, y_1, y_2, y_3, y_4]$, where the variables $y_1, y_2, y_3, y_4$ correspond to the generators of $\omega_R$. From the presentation of $\omega_R$, we get the ideal defining the relations involving generators of $\omega_R$, $I = \la -x_4y_1, -x_4y_2,-x_3y_1,-x_2y_1 + x_4y_3, x_1y_1-x_3y_2,-x_2y_2+x_4y_4, x_1y_2, -x_3y_3,-x_2y_3,x_1y_3-x_3y_4,-x_2y_4, x_1y_4\ra$. We see that the idealization $\tilde R$ is given by $T/J'$ where $J' = J_{\D} + {\la y_1, y_2, y_3, y_4 \ra}^2 + I$. It can be checked that $\tilde R$ fails WLP in degree $d-1 = 1$, as predicted by \cref{corGor}. We note that the Hilbert function of $\tilde R$ is $(1,8,8,1)$. Further, it can be checked using Macaulay2 that $J'$ has a Gr\"obner basis consisting of quadrics, and thus, $\tilde R$ is Koszul (see \cite[Proposition 2.12]{DV2021}).
\end{example}

\cref{Gor_ex} can be generalized as follows: Consider $\D$ to be any even cycle on $n=2a$ vertices for $a \ge 2$ and $R = \A= \k[x_1, \ldots, x_n]/J_{\D}$ (this ensures that $G(\D)$ is bipartite and hence, $R$ fails WLP in degree $1$). In the polynomial ring $T = \k[x_1, \ldots, x_n, y_1, \ldots, y_n]$, where $y_1, \ldots, y_n$ correspond to the generators of $\omega_R$, define ideal $I \subseteq T$ using the relations involving $\{y_1, \ldots, y_n\}$. To get a Gorenstein algebra $\tilde R$ in $2n$ variables, let $\tilde R = T/J'$ where $J' = J_{\D} + {\la y_1, \ldots, y_n \ra}^2 + I$. The Hilbert function of $\tilde R$ is $(1, 2n, 2n, 1)$. This $\tilde R$ fails WLP in degree $1$. This way, we get a family of Aritinian Gorenstein algebras in $2n$ variables that fail WLP in degree $1$ and having unimodal Hilbert function $(1,2n,2n,1)$, for $n \ge 4$.\\

We note that the resulting Gorenstein algebra constructed  in this section bear resemblance to the ones in \cite{GZ2018}, albeit from a rather different point of view. See also \cite[Section 8]{DV2021} for a detailed discussion of and some corrections to the results of \cite{GZ2018}.

\end{document}